\documentclass[12pt,leqno]{amsart}
\usepackage{subfigure}
\usepackage{epsfig}
\usepackage{amsmath, amsthm, amssymb}
\usepackage[colorlinks, citecolor=blue, linkcolor=red]{hyperref}

\newcommand{\be}{\begin{eqnarray}}
\newcommand{\ee}{\end{eqnarray}}

\newcommand{\RR}{\mathbb{R}}
\newcommand{\HH}{\mathbb{H}}
\theoremstyle{plain}

\newtheorem{theorem}{\textbf{Theorem}}[section]
\newtheorem{lemma}[theorem]{\textbf{Lemma}}
\newtheorem{proposition}[theorem]{\textbf{Proposition}}
\newtheorem{cor}[theorem]{\textbf{Corollary}}

\newtheorem{rem}[theorem]{\textbf{Remark}}
\theoremstyle{remark}

\newtheorem{ackn}{Acknowledgments\!\!}

\newcommand{\ricc}{\mathrm{Ric}}
\newcommand{\vol}{\mathrm{Vol}}
\newcommand{\laess}{\lambda_{1}^{\operatorname{ess}}}

\title[The Poisson equation on manifolds]
{The Poisson equation on manifolds\\ with positive essential spectrum}

\author[G. Catino]{Giovanni Catino}
\address[Giovanni Catino]{Dipartimento di Matematica, Politecnico di Milano, Piazza Leonardo da Vinci 32, 20133 Milano, Italy}
\email[]{giovanni.catino@polimi.it}

\author[D. D. Monticelli]{Dario D. Monticelli}
\address[Dario Monticelli]{Dipartimento di Matematica, Politecnico di Milano, Piazza Leonardo da Vinci 32, 20133 Milano, Italy}
\email[]{dario.monticelli@polimi.it}

\author[F. Punzo]{Fabio Punzo}
\address[Fabio Punzo]{Dipartimento di Matematica, Politecnico di Milano, Piazza Leonardo da Vinci 32, 20133 Milano, Italy}
\email[]{fabio.punzo@polimi.it}



\begin{document}

\begin{abstract}
We show existence of solutions to the Poisson equation on Riemannian manifolds with positive essential spectrum, assuming a sharp pointwise decay on the source function. In particular we can allow the Ricci curvature to be unbounded from below. In comparison with previous works, we can deal with a more general setting both on the spectrum and on the curvature bounds.
\end{abstract}

\maketitle

\section{Introduction}

In this paper we investigate existence of classical solutions to the Poisson equation
$$
-\Delta u = f
$$
on a complete non-compact Riemannian manifold $(M, g)$ for a given locally H\"older continuous function $f$. This very classical problem has been extensively studied under various assumptions on the geometry of the manifold, assuming either integral or pointwise conditions on the source function $f$.

As in \cite{ms}, \cite{msw}, \cite{nst}, we will be always concerned with solutions $u$ of the Poisson equation $-\Delta u=f$ which can be represented as
$$
u(x)=\int_{M} G(x,y)f(y)\,dy\,,
$$
where $G(x,y)$ is a Green's function on $M$ (see Section \ref{sect-prel} for further details).

In our results, the only geometric assumption that we require is that $(M,g)$ has {\em positive essential spectrum}, i.e. $\laess(M)>0$. We recall that this condition is weaker than the positivity of the whole $L^{2}$-spectrum of $-\Delta$, i.e. $\lambda_{1}(M)>0$, and it is not related to the {\em non-parabolicity} of the manifold, i.e. to the existence of a minimal positive Green's function $G(x,y)$ for $-\Delta$ (see Section \ref{sect-prel} for precise definitions). Moreover, we assume that the function $f$ satisfies a pointwise decay condition related to the geometry of the manifold at infinity. In addition, we show that such decay condition is optimal on spherically symmetric manifolds.

Concerning previous results in the literature, Malgrange \cite{mal} showed the existence of a Green's function for $-\Delta$ on every complete Riemannian manifolds, which implies the solvability of the Poisson equation for any smooth function $f$ with compact support (see \cite{tay} for a direct proof of this fact using pseudo-differential calculus).
Strichartz \cite{str} showed that if $(M,g)$ has positive spectrum and $f$ is a function belonging to $L^{p}(M)$ for some $1<p<\infty$, then the Poisson problem has a weak solution. The case $p=1$ was essentially proved by Ni-Shi-Tam \cite[Theorem 3.2]{nst} (see also \cite[Lemma 2.3]{ni}) only assuming $(M,g)$ non-parabolic. In the same paper, Ni-Shi-Tam proved a very nice existence result for the Poisson problem on manifolds with non-negative Ricci curvature under a sharp integral assumption involving suitable averages of $f$. This condition, in particular, is satisfied if
$$
|f(x)|\leq \frac{C}{\big(1+r(x)\big)^{2+\varepsilon}}
$$
for some $C,\varepsilon>0$, where $r(x):=\operatorname{dist}(x,p)$ is the distance function of any $x\in M$ from a fixed reference point $p\in M$. In fact, they proved a more general result where the decay rate of $f$ is just assumed to be of order $1+\varepsilon$. Note that this result is sharp on the flat space $\RR^{n}$ and that, in view of the nonnegative Ricci curvature assumption, $\lambda_{1}(M)=0$.


In \cite{ms} Muntenau-Sesum showed a very interesting existence result on manifolds with positive spectrum and Ricci curvature bounded from below, under the pointwise decay assumption
$$
|f(x)|\leq \frac{C}{\big(1+r(x)\big)^{1+\varepsilon}}
$$
for some $C,\varepsilon>0$. Note that this result is sharp on $\HH^{n}$. Their proof relies on very precise integral estimates on the minimal positive Green's function, which exists since $\lambda_{1}(M)>0$.

In order to state our results, let $\theta(R)$ be defined as in \eqref{eq127} (see Section \ref{sect-prel} for further details). For the moment we just note that such a function is related to a lower bound on the Ricci curvature, locally on geodesic balls with center $p$ and radius $R$.

The main result of the paper reads as follows.

\begin{theorem} \label{teo1} Let $(M,g)$ be a complete non-compact Riemannian manifold with positive essential spectrum, i.e. $\laess(M)>0$, and let $f$ be a locally H\"older continuous function on $M$ satisfying
\begin{equation*}
| f (x) | \leq \frac{1}{\zeta\left(r(x)\right)} \quad\hbox{for all  } x\in M,
\end{equation*}
for some non-decreasing function $\zeta\in C^{0}([0,\infty);(0,\infty))$. If
$$
\sum_j^{\infty}\frac{\theta(j+1)-\theta(j)}{\lambda_{1}\left(M\setminus B_{j-1}(p)\right)\zeta(j-1)} < \infty,
$$
then the Poisson equation
\begin{equation*}
-\Delta u=f \quad\hbox{in   } M
\end{equation*}
admits a classical solution $u$.
\end{theorem}
A few remarks concerning the assumptions are in order.

Since $\laess(M)\geq\lambda_{1}(M)$, it is clear that the result applies to a wider class of manifolds than those with positive spectrum. For instance we can deal with manifolds with positive essential spectrum and finite volume (e.g. hyperbolic manifolds with finite volume).

The quantities $\theta(j+1)-\theta(j)$ and $\lambda_{1}\left(M\setminus B_{j-1}(p)\right)$ are related to the geometry of the manifold at infinity. In particular, if $(M,g)$ has Ricci curvature bounded from below, $\ricc \geq -K$, then $\theta(j+1)-\theta(j) \leq C$ for any $j$. Moreover, by monotonicity, $\lambda_{1}\left(M\setminus B_{j-1}(p)\right) \geq \lambda_{1}(M)$. Thus, if $(M,g)$ satisfies $\lambda_{1}(M)>0$, $\ricc\geq -K$ and $f$ decays like $C/\big(1+r(x)\big)^{1+\varepsilon}$, for some $C,\varepsilon>0$, then all our assumptions are satisfied and we recover the existence result of Munteanu-Sesum's \cite{ms}. Note that our proof does not provide any information on the behavior of the solution at infinity, contrary to the aforementioned result. 

We also note that the assumptions of Theorem \ref{teo1} are sharp as shown by the example of a rotationally symmetric manifold with unbounded Ricci curvature,  provided in Section \ref{sec-ex}.

Concerning the regularity of $f$, we observe that, if $f$ is less regular, the same existence result holds for weak solutions.

In general, we can provide existence of solutions also on manifolds with unbounded Ricci curvature. For instance, we have the following

\begin{cor}\label{cor-1} Let $(M,g)$ be a complete non-compact Riemannian manifold with positive essential spectrum, i.e. $\laess(M)>0$, and let $f$ be a locally H\"older continuous function on $M$.  If
\begin{equation*}
\ricc\geq -C\big(1+r(x)\big)^{\gamma},\quad\quad |f (x)| \leq \frac{C}{\big(1+r(x)\big)^{1+\frac{\gamma}{2}+\varepsilon}},
\end{equation*}
for some $C>0$, $\gamma \geq 0$ and $\varepsilon>0$, then the Poisson equation
\begin{equation*}
-\Delta u=f \quad\hbox{in   } M
\end{equation*}
admits a classical solution $u$.
\end{cor}

Using a Barta-type estimate, we can show that on Cartan-Hadamard manifolds with $\ricc\leq -(1/C)(1+r)^{\gamma}$ for some $\gamma\geq0$, one has
$\lambda_{1}\left(M\setminus B_{j-1}(p)\right) \geq C \,j^{\gamma}$. So, in particular, $\lambda^{\operatorname{ess}}_{1}(M)>0$ (actually $\lambda_{1}(M)>0$) and we have the following corollary which generalizes the existence result in \cite[Proposition 1]{ms}.

\begin{cor}\label{cor-2} Let $(M,g)$ be a Cartan-Hadamard manifold and let $f$ be a locally H\"older continuous, bounded function on $M$. If
$$
-C\big(1+r(x)\big)^{\gamma_{1}} \leq \ricc \leq -\frac{1}{C}\big(1+r(x)\big)^{\gamma_{2}} ,\quad |f (x)| \leq \frac{C}{\big(1+r(x)\big)^{1+\frac{\gamma_{1}}{2}-\gamma_{2}+\varepsilon}},
$$
for some $C>0$, $\gamma_{1},\gamma_{2}\geq 0$ and $\varepsilon>0$ with $1+\frac{\gamma_{1}}{2}-\gamma_{2}+\varepsilon\geq 0$, then the Poisson equation
\begin{equation*}
-\Delta u=f \quad\hbox{in   } M
\end{equation*}
admits a classical solution $u$.
\end{cor}
In particular, Corollary \ref{cor-2} implies that, if $\gamma_{1}=\gamma_{2}=\gamma>2$, existence of a solution to the Poisson equation is guaranteed whenever $f$ is bounded.
In the case $\gamma>2$, Cartan-Hadamard manifolds behave, in some sense, like bounded domains of $\RR^{n}$.

When the ambient manifold is non-parabolic, our approach follows the one in \cite{ms} which in turn uses some ideas originated in the work of Li-Wang \cite{liwa1}. In the other case, thanks to the works \cite{csz, liwa2} all ends of the manifold must have finite volume and thus we can use the decay estimates of Donnelly \cite{do} on general Green's functions to conclude.

Indeed, the proof of Theorem \ref{teo1} essentially consists in obtaining an a priori pointwise bound on the solution $u$ represented by using the Green's function. In order to do this, we need several estimates on the Green function. First a pointwise two-sided estimate is obtained, which is deduced from a local gradient estimate for harmonic functions. Then we get integral estimates from above for $G$. In particular, integral estimates on certain level sets will be derived. In such bounds the behaviour at infinity of Ricci curvature is involved. Finally, putting together conveniently such estimates and using the hypothesis that $\lambda^{\operatorname{ess}}_{1}(M)>0$, we can prove Theorem \ref{teo1}, by showing that for every $x\in M$
\[\left|\int_M G(x,y) f(y) dy \right| < \infty\,. \]

The paper is organized as follows: in Section \ref{sect-prel} we collect some preliminary results and we define precisely the function $\theta$; in Section \ref{sec-grad} we recall and prove (for the sake of completeness) the local gradient estimates for positive harmonic functions; in Section \ref{sec-est} we prove key estimates on the positive minimal Green's function $G(x,y)$ of a non-parabolic manifold; in Section \ref{sec-proofs} we prove Theorem \ref{teo1} and Corollary \ref{cor-1}; finally in Section \ref{sec-ex} we prove Corollary \ref{cor-2} and show the optimality of the assumption in Theorem \ref{teo1} for rotationally symmetric manifolds.

\

\section{Preliminaries}

\label{sect-prel}

Let $(M,g)$ be a complete non-compact $n$-dimensional Riemannian manifold. For any $x\in M$ and $R>0$, we denote by $B_{R}(x)$ the geodesic ball of radius R with centre $x$ and let $\vol(B_{R}(x))$ be its volume. We denote by $\ricc$ the Ricci curvature of $g$. For any $x \in M$, let $\mu(x)$ be the smallest eigenvalue of $\ricc$ at $x$. Thus, for any $V\in T_{x}M$ with $|V|=1$, $\ricc(V,V)(x) \geq \mu(x)$ and we have $\mu(x)\geq -\omega(r(x))$ for some $\omega\in C([0,\infty))$, $\omega\geq 0$. Hence, for any $x\in M$, we have
\begin{equation}\label{eq3}
\ricc(V,V)(x) \geq -(n-1) \frac{\varphi''(r(x))}{\varphi(r(x))},
\end{equation}
for some $\varphi\in C^{\infty}((0,\infty))\cap C^{1}([0,\infty))$ with $\varphi(0)=0$ and $\varphi'(0)=1$. Note that $\varphi,\varphi',\varphi''$ are positive in $(0,\infty)$, and $\frac{\varphi''}{\varphi}\in C([0, +\infty))$. Fix any $\varepsilon_0>0$ so that 
\[\Delta r(x) \geq 0\quad \textrm{in}\;\; B_{4\varepsilon_0}(p)\,.\]
For any $R>\varepsilon_0$ we set
\begin{align}
\tilde{K}(R) &:= \sup_{y\in B_{R}(p)\setminus B_{\varepsilon_0}(p)} \frac{\varphi''(r(y))}{\varphi(r(y))}, \nonumber\\
\nonumber \hat{K}(R) &:= \sup_{y\in B_{R}(p)\setminus B_{\varepsilon_0}(p)} \frac{\varphi'(r(y))}{\varphi(r(y))}, \\ \label{eq5}
K(R)&:= \max\{1,\tilde{K}_{R}(x), \hat{K}_{R}(x) \}.
\end{align}
For any $R>1$ we define
\begin{equation}\label{eq127}
\theta(R):=R \sqrt{K(R)}.
\end{equation}
Note that $R\mapsto\theta(R)$ is increasing and so invertible.

Under $\eqref{eq3}$, we know that
\begin{equation}\label{eq6}
\vol(B_{R}(p)) \leq C \int_{0}^{R}\varphi^{n-1}(\xi)\,d\xi.
\end{equation}
Moreover, let $\operatorname{Cut}(p)$ be the {\em cut locus} of $p\in M$. By standard Laplacian comparison results (see e.g. \cite[Theorem 1.11]{masrigset}),
\begin{equation}\label{eq7}
\Delta r(x) \leq (n-1)\frac{\varphi'(r(x))}{\varphi(r(x))}
\end{equation}
pointwise in $M\setminus (\{p\}\cup \operatorname{Cut}(p))$ and weakly on $M$.

It is known that every complete Riemannian manifold admits a Green's function (see \cite{mal}), i.e. a smooth function defined in $(M\times M)\setminus \{(x,y)\in M\times M:\,x=y\} $ such that $G(x,y)=G(y,x)$ and $\Delta_{y} G(x,y)=-\delta_{x}(y)$. We say that $(M,g)$ is non-parabolic if there exists a minimal positive Green's function $G(x,y)$ on $(M,g)$, and parabolic otherwise.

Let $\lambda_{1}(M)$ be the bottom of the $L^{2}$-spectrum of $-\Delta$. It is known that $\lambda_{1}(M)\in[0,+\infty)$ and it is given by the variational formula
$$
\lambda_{1}(M) = \inf_{v\in C^{\infty}_{c}(M)}\frac{\int_{M}|\nabla v|^{2}\,dV}{\int_{M}v^{2}\,dV}\,.
$$
If $\lambda_{1}(M)>0$, then $(M,g)$ is non-parabolic (see \cite[Proposition 10.1]{gri1}). Whenever $(M,g)$ is non-parabolic, let $G_{R}(x,y)$ be the Green's function of $-\Delta$ in $B_{R}(z)$ satisfying zero Dirichlet boundary conditions on $\partial B_{R}(z)$, for some $z\in M$. We have that $R\mapsto G_{R}(x,y)$ is increasing and, for any $x,y\in M$,
\begin{equation}\label{eq9}
G(x,y) = \lim_{R\to\infty} G_{R}(x,y),
\end{equation}
locally uniformly in $(M\times M)\setminus \{(x,y)\in M\times M:\,x=y\} $.
We define $\lambda_{1}(\Omega)$, with $\Omega$ an open subset of $M$, to be the first eigenvalue of $-\Delta$ in $\Omega$ with zero Dirichlet boundary conditions. It is well known that $\lambda_{1}(\Omega)$ is decreasing with respect to the inclusion of subsets. In particular $R\mapsto\lambda_{1}(B_{R}(x))$ is decreasing and $\lambda_{1}(B_{R}(x))\to \lambda_{1}(M)$ as $R\to\infty$.

Another interesting quantity of $(M,g$), which we denote by $\laess(M)$, is the greatest lower bound of the essential spectrum of $-\Delta$, which consists of points of the spectrum of $-\Delta$ which are either accumulation points of points on the spectrum or which correspond to discrete eigenvalues of $-\Delta$ with infinite multiplicity. It is known that if $(M,g)$ is compact, the essential spectrum is empty. We also have $\lambda_{1}(M)\leq \laess(M)$ and
$$
\laess(M)=\sup_{K} \lambda_{1}(M\setminus K),
$$
where $K$ runs through all compact subsets of $M$.

\

\

We explicitly note that by $C$ we will denote a positive constant, whose value could vary.

%
%
%
%
%
%

\

\section{Local gradient estimate for harmonic functions} \label{sec-grad}
Following the classical argument of Yau, we obtain the next local gradient estimate. 

\begin{lemma}\label{lemma0} Fin any $\varepsilon_0>0$. Let $R+1>3\varepsilon_0$ and let $u\in C^{2}(B_{R+1}(p)\setminus B_{\varepsilon_0}(p))$ be a positive harmonic function in $B_{R+1}(p)\setminus B_{\varepsilon_0}(p)$. Then
$$
|\nabla u(\xi)| \leq C \sqrt{K(R+1)}\, u(\xi)\quad\text{for any}\quad \xi\in B_{R}(p)\setminus B_{3\varepsilon_0}(p),
$$
with $K(R)$ defined as in \eqref{eq5} and for some positive constant $C=C(\varepsilon_0)>0$.
\end{lemma}
\begin{proof}
Let $v:=\log u$. Then
$$
\Delta v = - |\nabla v|^{2} .
$$
Let $\eta(\xi)=\eta(\rho(\xi))$, with $\rho(\xi):=\operatorname{dist}(\xi,p)$, be a smooth cut-off function, with support in $B_{R+1}(p)\setminus B_{2\varepsilon_0}(p)$, $0\leq \eta\leq 1$, such that $\eta(\xi)\equiv 1$ on $B_{R}(p)\setminus B_{3\varepsilon_0}(p)$, $$-C\leq \frac{\eta'}{\eta^{1/2}} \leq 0 \quad\text{and}\quad \frac{|\eta''|}{\eta} \leq C\quad \textrm{in}\;\; B_{R+1}(p)\setminus B_{R}(p),$$ 
and 
$$0\leq \frac{\eta'}{\eta^{1/2}} \leq \frac{2}{\varepsilon_0}, \quad\text{and}\quad \frac{|\eta''|}{\eta} \leq \frac{2}{\varepsilon_0^{2}}\quad \textrm{in}\;\; B_{3\varepsilon_0}(p)\setminus B_{2\varepsilon_0}(p).$$

Let $w=\eta^{2}|\nabla v|^{2}$. Then
\begin{align*}
\frac12 \Delta w &= \frac12 \eta^{2} \Delta |\nabla v|^{2} + \frac12 |\nabla v|^{2} \Delta \eta^{2} + \langle \nabla |\nabla v|^{2},\nabla \eta^{2}\rangle.
\end{align*}
Hence, from classical Bochner-Weitzenb\"och formula and Newton inequality, one has
\begin{align*}
\frac12 \Delta |\nabla v|^{2} & =  |\nabla^{2} v|^{2} + \ricc(\nabla v,\nabla v) + \langle \nabla v,\nabla \Delta v\rangle \\
&\geq \frac1n (\Delta v)^{2} - (n-1) \frac{\varphi''}{\varphi} |\nabla v|^{2} - \langle \nabla |\nabla v|^{2},\nabla v\rangle \\
&= \frac1n |\nabla v|^{4} - (n-1) \frac{\varphi''}{\varphi} |\nabla v|^{2} - \langle \nabla |\nabla v|^{2},\nabla v\rangle .
\end{align*}
Moreover, let $\chi_R$ be the characteristic function of the set $B_{R+1}(p)\setminus B_{R}(p)$, and let $\chi_{\varepsilon_0}$ be the characteristic function of the set $B_{3\varepsilon_0}(p)\setminus B_{2\varepsilon_0}(p)$. By \eqref{eq7},
\begin{align*}
\frac12 \Delta \eta^{2} &= \eta \eta' \Delta \rho + \eta \eta'' + (\eta')^{2} \\
&\geq (n-1)\frac{\varphi'}{\varphi}\eta\eta'\chi_R + \eta \eta''+ (\eta')^{2}\\
&\geq -C \left((n-1)\frac{\varphi'}{\varphi}+1\right)\eta \chi_R - \frac{2}{\varepsilon_0^2}\eta \chi_{\varepsilon_0}
\end{align*}
pointwise in $[B_{R+1}(p)\setminus B_{\varepsilon_0}(p)]\setminus (\{z\}\cup \operatorname{Cut}(p))$ and weakly on $B_{R+1}(p)\setminus B_{\varepsilon_0}(p)$ with $C>0$. Thus,
\begin{align*}
\frac12 \Delta w &\geq \frac1n \frac{w^{2}}{\eta^{2}}-(n-1)\frac{\varphi''}{\varphi}w - C\left((n-1)\frac{\varphi'}{\varphi}+1\right)\frac{w}{\eta}\chi_R-\frac{2}{\varepsilon_0^2}\frac{w}{\eta} \chi_{\varepsilon_0} \\
&-4\frac{|\eta'|^{2}}{\eta^{2}}w + \frac{2}{\eta}\langle \nabla w,\nabla \eta\rangle-\langle \nabla w,\nabla v\rangle + \frac{2}{\eta}\langle \nabla v,\nabla \eta \rangle w \\
&\geq \frac1n \frac{w^{2}}{\eta^{2}}-(n-1)\frac{\varphi''}{\varphi}w - C'\left((n-1)\frac{\varphi'}{\varphi}+1\right)\frac{w}{\eta}\chi_R -\frac{2}{\varepsilon_0^2}\frac{w}{\eta} \chi_{\varepsilon_0} \\
& + \frac{2}{\eta}\langle \nabla w,\nabla \eta\rangle-\langle \nabla w,\nabla v\rangle-C\frac{ w^{3/2}}{\eta^{3/2}}\\
&\geq \frac{1}{2n} \frac{w^{2}}{\eta^{2}}-(n-1)\frac{\varphi''}{\varphi}w - C''\left(\chi_R(n-1)\frac{\varphi'}{\varphi}+1\right)\frac{w}{\eta} \\
& + \frac{2}{\eta}\langle \nabla w,\nabla \eta\rangle-\langle \nabla w,\nabla v\rangle\,,
\end{align*}
with $C''>0$. Let $q$ be a maximum point of $w$ in $\overline{B}_{R+1}(p)\setminus B_{\varepsilon_0}(p)$. Since $w\equiv0$ on $\partial B_{R+1}(p)\cup \partial B_{\varepsilon_0}(p)$, we have  $q\in B_{R+1}(p)\setminus \overline{B}_{\varepsilon_0}(p)$. First assume $q\notin \operatorname{Cut}(p)$. At $q$, we obtain
\begin{align*}
0 &\geq \left[\frac{1}{2n} w - (n-1)\frac{\varphi''}{\varphi}-C'\Big(\frac{\varphi'}{\varphi}\chi_R+1\Big)\right]w.
\end{align*}
So,
$$
w(q)\leq 2n(n-1)\big(1+C'\big) \left(1+\frac{\varphi'\big(r(q)\big)}{\varphi\big(r(q)\big)}\chi_R(q)+\frac{\varphi''\big(r(q)\big)}{\varphi\big(r(q)\big)}\right).
$$
Thus, for any $\xi \in B_{R}(p)\setminus B_{3\varepsilon_0}(p)$,
$$
|\nabla v(\xi)|^{2}\leq 2n(n-1)(1+C') \left(1+\frac{\varphi'\big(r(q)\big)}{\varphi\big(r(q)\big)}\chi_R(q)+\frac{\varphi''\big(r(q)\big)}{\varphi\big(r(q)\big)}\right).
$$
Finally, we get
$$
\frac{|\nabla u(\xi)|}{u(\xi)}=|\nabla v(\xi)| \leq C \sqrt{K(R+1)}.
$$
with $C>0$. By the standard Calabi trick (see \cite{cal, cy}), the same estimate can be obtained when $q\in \operatorname{Cut}(p)$. This concludes the proof of the lemma.

\end{proof}

\

\section{Green's function estimates} \label{sec-est}
This section is devoted to various estimates for the Green function. In particular, we recall from \cite{do} an estimate for
\[\int_{M\setminus B_R(x)} [G(x,y)]^2 dy\,,\]
for any $x\in M$, which holds without assuming that $G$ is positive.

Furthermore, for any $0\leq a < b\leq +\infty$, let
\begin{align*}
\mathcal{L}(a,b)&:= \{y \in M\,:\, a< G(p,y)< b \}.
\end{align*}
We obtain estimates from above (see subsection \ref{lsest}) for
\[\int_{\mathcal{L}(a,b)} G(p,y) dy\,,\]
for suitable choices of $a, b.$

\subsection{Pointwise estimate}\label{pointest}

By the continuity of $y\mapsto G(p, y)$ in $M\setminus\{p\}$, obviously we have the following
\begin{lemma}\label{lemma1} Let $(M,g)$ be non-parabolic and let $y\in M$ with $y\in \partial B_{1}(p)$. Then there exists a positive constants $A>1$ such that
$$
A^{-1}\leq G(p,y) \leq A .
$$
\end{lemma}

\begin{rem}\label{remark100} Indeed, the estimate from below given in Lemma \ref{lemma1} holds for any any $y\in \overline{B_{1}(p)}$. This follows from Lemma \ref{lemma1} with $y\in \partial B_{1}(p)$ and the maximum principle, since $y\mapsto G(p,y)$ is harmonic in $B_{1}(p)\setminus\{p\}$ and $G(p,y)\to \infty$ as $y\to p$.
\end{rem}

\begin{lemma}\label{remark200}
Let $z\in M\setminus B_{1}(p)$, and consider the minimal unit speed geodesic $\gamma$ joining $p$ and $z$, and let $z_{0}\in\partial B_{1}(p)$ be a point of intersection of $\gamma$ with $\partial B_{1}(p)$. Then
$$
G(p,z) \geq G(p,z_{0}) \exp \left(-C_{0}\sqrt{K(r(z)+1)}r(z)\right)\,.
$$
\end{lemma}

\begin{proof} Fix any $\varepsilon_0\in \left(0, \frac 1{2} \right)$. By Lemma \ref{lemma0} we get for every $\xi\in \gamma$
$$
|\nabla G(p,\xi)| \leq C \sqrt{K(r(z)+1)}\,G(p,\xi) .
$$
Let $l(\gamma)$ be the length of $\gamma$. We have
\begin{align*}
G(p,z)&=G(p,z_0)+\int_{1}^{l(\gamma)}\langle \nabla G(p,\gamma(s)), \dot{\gamma}(s)\rangle \,ds \\
&\geq G(p,z_0) - C\sqrt{K(r(z)+1)}\int_{1}^{l(\gamma)} G(p,\gamma(s)) \,ds.
\end{align*}
By Gronwall inequality the conclusion follows.

\end{proof}

\subsection{Auxiliary estimates}\label{auxest}
For any $s>0$, let
\begin{align*}
\mathcal{L}(s) &:= \{y \in M\,:\,G(p,y)=s \}.
\end{align*}

\begin{lemma}\label{lemma2} Let $(M,g)$ be non-parabolic. For any $s>0$, there holds
$$
\int_{\mathcal{L}(s)}|\nabla G(p,y)|\,dA(y) = 1
$$
where $dA(y)$ is the $(n-1)$-dimensional Hausdorff measure on $\mathcal{L}(s)$.
\end{lemma}
For the proof see \cite[Lemma 2]{ms}. Moreover, on manifolds with positive essential spectrum, the following decay estimate holds (see  \cite[Section 4]{do} and also \cite[Corollary 1.3]{liwa2} for a sharp estimate in the case of positive spectrum).

\begin{lemma}\label{lemmadonnelly} Assume $\laess(M)>0$ and let $0<\beta<\laess(M)$. Then, for any $x\in M, R\geq 1$, one has
\begin{itemize}

\item[(i)] If $(M,g)$ is non-parabolic, then the minimal positive Green's function $G(x,y)$ satisfies 
$$
\int_{M\setminus B_{R}(x)} G(x,y)^{2}\,dy \leq C \exp\left(-2\sqrt{\beta}\,R\right);
$$

\item[(ii)] If $V:=\vol(M)<\infty$, then the (sign changing) Green's function $\bar{G}(x,y)$
\begin{equation}\label{greendon}
\bar{G}(x,y):=\int_0^\infty \left(p(x,y,t)-\frac{1}{V}\right)dt
\end{equation}
satisfies $-\Delta \bar{G}(\cdot,y)=\delta_y-\frac{1}{V}$ and
$$
\int_{M\setminus B_{R}(x)} |\bar{G}(x,y)|^{2}\,dy \leq C \exp\left(-2\sqrt{\beta}\,R\right).
$$
\end{itemize}
\end{lemma}

\subsection{Integral estimates on level sets}\label{lsest}

\begin{proposition}\label{lemma3} Let $(M,g)$ be non-parabolic. Choose $A$ as in Lemma \ref{lemma1}. Then there exists a positive constant $C$ such that
\begin{align*}
\int_{\mathcal{L}\left(A,\infty\right)} &G(p,y)\,dy \leq \,  C.
\end{align*}
\end{proposition}
\begin{proof}
We claim that
$$
\mathcal{L}_{p}\left(A,\infty\right) \subset B_{1}(p).
$$
Let $y\in M$ with $\operatorname{dist}(p,y)>1$ and take $j>\operatorname{dist}(p,y)$. Since $G_{j}(p,y)\leq G(p,y)$ and $G_{j}(p,\cdot)\equiv 0$ on $\partial B_{j}(p)$, by Lemma \ref{lemma1}, we have
$$
G_{j}(p,y)\leq A \quad\text{on}\quad \partial\left(B_{j}(p)\setminus B_{1}(p)\right);
$$
note that the right hand side is independent of $y$. Since $y\mapsto G_{j}(p,y)$ is harmonic in $B_{j}(p)\setminus B_{1}(p)$, by maximum principle,
$$
G_{j}(p,y)\leq A \quad\text{in}\quad B_{j}(p)\setminus B_{1}(p).
$$
Sending $j\to\infty$, by \eqref{eq9}, we obtain
$$
G(p,y)\leq A\quad\text{in}\quad M\setminus B_{1}(p).
$$
Therefore the claim follows. Since $G(p, \cdot)\in L^1_{\textrm{loc}}(M)$, we obtain
\begin{equation}\label{in1}
\int_{\mathcal{L}\left(A,\infty\right)} G(p,y)\,dy \leq \int_{B_{1}(p)} G(p,y)\,dy=: C<+\infty.
\end{equation}
\end{proof}

Note that, for any $R>0$, $0<\delta<1$ and $\varepsilon >0$, if $\phi\in C^{\infty}_{c}(M)$ with $\operatorname{supp} \phi \subseteq \left(\mathcal{L}\Big(\frac{\delta\varepsilon}{2},2\varepsilon\Big)\cap B_{R+1}(p)\right)$ and $\phi\equiv 1$ on $\mathcal{L}(\delta\varepsilon,\varepsilon)\cap B_{R}(p)$ then
\begin{align*}
\lambda_{1}\left(\mathcal{L}\Big(\frac{\delta\varepsilon}{2},2\varepsilon\Big)\right)&\int_{\mathcal{L}(\delta\varepsilon,\varepsilon)\cap B_{R}(p)}G(p,y)\,dy \\
&\leq \lambda_{1}\left(\mathcal{L}\Big(\frac{\delta\varepsilon}{2},2\varepsilon\Big)\cap B_{R+1}(p)\right)\int_{M}G(p,y)\phi(y)^{2}\,dy \\
&\leq \int_{M}|\nabla G^{\frac{1}{2}}(p,y)\phi(y)|^{2}\,dy.
\end{align*}
In view of the previous estimate, by the same arguments as in the proof of \cite[Claim 2]{ms} (which do not require that the Ricci curvature is bounded from below by a negative constant) and using Lemma \ref{lemma2} and Lemma \ref{lemmadonnelly} $(i)$, we get the next estimate.

\begin{proposition}\label{claim2} Assume $(M,g)$ be non-parabolic and $\laess(M)>0$. Then, there exists a positive constant $C$ such that, for any $0<\delta<1$ and $\varepsilon >0$,
$$
\lambda_{1}\left(\mathcal{L}\Big(\frac{\delta\varepsilon}{2},2\varepsilon\Big)\right)\int_{\mathcal{L}(\delta \varepsilon, \varepsilon)} G(p,y)\,dy \leq C \left(-\log\delta +1\right).
$$
\end{proposition}

\

\section{Proof of Theorem \ref{teo1}} \label{sec-proofs}

In order to prove Theorem \ref{teo1}, we will show that for every $x\in M$
$$
|u(x)|=\left| \int_{M}G(x,y)f(y)\,dy \right| <+\infty\,.
$$
We divide the proof in two parts, we first consider the case when $(M,g)$ is non-parabolic and then the case when it is parabolic.

\begin{proof}[Proof of Theorem \ref{teo1}] {\bf Case 1:} {\em  $(M,g)$ non-parabolic.} Let $x\in M$ and choose $R=R(x)>0$ large enough so that $x\in B_R (p)$. One has

\begin{align*}
\left|\int_M G(x,y)\,f(y)\, dy\right| &\leq \left| \int_{B_R(p)} G(x,y)\,f(y)\,dy \right|+\left|\int_{M\setminus B_R(p)} G(x,y)\,f(y)\,dy\right|\\
&\leq C_1(x) + \int_{M\setminus B_R(p)} G(x,y)\,|f(y)|\,dy
\end{align*}
since $G(x,\cdot)\in L^1_{\text{loc}}(M)$.  Hence, by Harnack's inequality, we have
\begin{align}\label{eq501}
\left|\int_M G(x,y)\,f(y)\, dy\right| &\leq C_1(x) + C_2(x)\int_{M\setminus B_R(p)} G(p,y)\,|f(y)|\,dy \\ \nonumber
&\leq C_1(x) + C_2(x)\int_{M} G(p,y)\,|f(y)|\,dy \,,
\end{align}
where $C_2(x)$ can be chosen as the constant in the Harnack's inequality for the ball $B_{r(x)+1}(p)$. Then we estimate

\begin{align*}
 \int_{M}G(p,y)\,|f(y)|\,dy   &= \int_{\mathcal{L}\left(0, \,A \right)} G(p,y)\,|f(y)|\,dy  \\
 &\,\,\,+  \int_{\mathcal{L}\left(A ,\infty\right)} G(p,y)\,|f(y)|\,dy \,.
\end{align*}
By Proposition \ref{lemma3}, we get
\begin{align*}
 \int_{M}G(p,y)\,|f(y)|\,dy   &\leq \int_{\mathcal{L}\left(0, A\right)} G(p,y)\,|f(y)|\,dy  + C\,.
\end{align*}

For later use, we introduce the sequence
\[a_m:=\frac{\exp\left(-C_0 \theta( m)\right)}{2 A}\,.\]
To estimate the first term, we observe that, for any $m_{0}\geq 2$, one has
\begin{align*}
\left| \int_{\mathcal{L}\left(0,\,A\right)} G(p,y)f(y)\,dy \right| \leq \left| \int_{\mathcal{L}\left(0, \,a_{m_0}\right)}G(p,y)f(y)\,dy \right| \\
+\left| \int_{\mathcal{L}\left(a_{m_0},\,A \right)}G(p,y)f(y)\,dy \right|\,.
\end{align*}
We need the following lemma, whose proof will be given at the end of this section.
\begin{lemma}\label{lemma5} Choose $A$ as in Lemma \ref{lemma1} and $C_{0}$ as in Lemma \ref{remark200}.
Then, for any $m\in \mathbb N$,
\begin{equation}\label{eq400}\mathcal{L}\left(0, \, 2 a_m\right) \subset M \setminus B_{m-1}(p).
\end{equation}
\end{lemma}

Fix any $m_0\in \mathbb N, m_0\geq 2$. By the same arguments as in the proof of Proposition \ref{lemma3}, we get
$$
\mathcal{L}\left(a_{m_0},\, A \right) \subset B_{l}(p)
$$
for some $l=l(m_0)>0,$ and
\begin{align}\label{eq200}
\left| \int_{\mathcal{L}\left(a_{m_0},\,A \right)}G(p,y)f(y)\,dy \right|\leq C\,.
\end{align}

Now, for any $m\geq m_{0}$, applying Lemma \ref{claim2} with $\varepsilon:=a_m$ and $\delta:=\frac{a_{m+1}}{a_m}$, we obtain
\begin{align}\label{eq201}
&\Big| \int_{\mathcal{L}\left(0, \,a_{m_0}\right)}G(p,y)f(y)\,dy \Big| \leq \sum_{m\geq m_{0}}\left| \int_{\mathcal{L}\left(a_{m+1}, \,a_m\right)}G(p,y)f(y)\,dy \right| \\\nonumber
&\leq C \sum_{m\geq m_{0}} \frac{\theta(m+1)-\theta(m)}{\lambda_{1}\left(\mathcal{L}\left(\frac{1}{2}a_{m+1}, \,2a_m\right)\right)}\sup_{\mathcal{L}\left(a_{m+1},a_m\right)}|f|\\\nonumber
&\leq C \sum_{m\geq m_{0}} \frac{\theta(m+1)-\theta(m)}{\lambda_{1}\left(\mathcal{L}\left(\frac{1}{2}a_{m+1}, \,2a_m\right)\right)}\sup_{\mathcal{L}\left(\frac 1 2a_{m+1}, 2 a_m\right)}|f|\,.
\end{align}
Using the assumptions of Theorem \ref{teo1}, from inequality \eqref{eq201}, Lemma \ref{lemma5} and  the monotonicity of $\lambda_{1}$ with respect to the inclusion we obtain
\begin{align}\label{est2}
\Big| &\int_{\mathcal{L}\big(0, a_{m_0}\big)}G(p,y)f(y)\,dy \Big| \\\nonumber
&\leq
C \sum_{m\geq m_{0}} \frac{\theta(m+1)-\theta(m)}{\lambda_{1}\left(M\setminus B_{m-1}(p)\right)}\sup_{M\setminus B_{m-1}(p)}|f|\\\nonumber
&\leq C \sum_{m\geq m_{0}} \frac{\theta(m+1)-\theta(m)}{\lambda_{1}\left(M\setminus B_{m-1}(p)\right)\zeta\left(m-1\right)}\\\nonumber
&\leq C \sum_m^\infty\frac{\theta(m+1)-\theta(m)}{\lambda_{1}\left(M\setminus B_{m-1}(p)\right)\zeta\left(m-1\right)}<\infty.
\end{align}

Putting together \eqref{eq200}, \eqref{est2} we obtain the thesis of Theorem \ref{teo1}, in this case.

\

\noindent {\bf Case 2:} {\em  $(M,g)$ parabolic.}

\

Since $\laess(M)>0$, there exists $\bar{R}>0$ such that $\lambda_{1}(M\setminus B_{\bar{R}}(p))>0$. Fix any $x\in M$. We can choose $R_{0}\geq r(x)+1$ such that $B_{R_{0}}(x)\supset B_{\bar{R}}(p)$. So,
$$
\lambda_{1}\left(M\setminus B_{R_{0}}(x)\right) \geq \lambda_{1}\left(M\setminus B_{\bar{R}}(p)\right)>0.
$$
We have that
$$
M\setminus B_{R_{0}}(x) = \bigcup_{i=1}^{N} E_{i},
$$
where each $E_{i}$ is an end with respect to $B_{R_{0}}(x)$. Note that every end $E_{i}$ is parabolic. In fact, if at least one end $E_{i}$ is non-parabolic, then $(M,g)$ is non-parabolic (see \cite{li} for a nice overview), but we are in the case that $(M,g)$ is parabolic. Since every $E_{i}$ is parabolic, every $E_{i}$ has finite volume (see \cite{csz, liwa2}), so obviously $(M,g)$ has finite volume. Let $\bar{G}(x,y)$ be the sign changing Green's function defined in \eqref{greendon}. 

Let $\varphi\in C^\infty_c(M)$ a smooth compactly supported function with $\int_M \varphi =1$. Let $\psi\in C^\infty(M)$ be a solution of
$$
-\Delta \psi = \varphi
$$
and let $\bar{f}:=f-\alpha \varphi$, with $\alpha$ fixed so that $\int_M \bar{f}=0$. Then
Hence,
\begin{align*}
\left| \int_{M}\bar{G}(x,y)\bar{f}(y)\,dy\right| &\leq \int_{M\setminus B_{R_{0}}(x)}|\bar{G}(x,y)||\bar{f}(y)|\,dy + \int_{B_{R_{0}}(x)}|\bar{G}(x,y)||\bar{f}(y)|\,dy \\
&\leq \int_{M\setminus B_{R_{0}}(x)}|\bar{G}(x,y)||\bar{f}(y)|\,dy+  C,
\end{align*}
since $\bar{G}(x,\cdot)\in L^{1}_{\rm{loc}}(M)$ and $\bar{f}$ is bounded. To estimate the first integral we recall that, by Lemma \ref{lemmadonnelly} $(ii)$ we have
\begin{align*}
\int_{M\setminus B_{R_{0}}(x)}|\bar{G}(x,y)||\bar{f}(y)|\,dy &\leq \left(\int_{M\setminus B_{R_{0}}(x)}|\bar{G}(x,y)|^{2}\,dy\right)^{\frac{1}{2}}\left(\int_{M\setminus B_{R_{0}}(x)}|\bar{f}(y)|^{2}\,dy\right)^{\frac{1}{2}}\\
&\leq C\, \vol\left(M\setminus B_{R_{0}}(x)\right)^{\frac{1}{2}}\Vert \bar{f} \Vert_{L^{\infty}(M)} \exp\left(-\sqrt{\beta}\,R_{0}\right)<\infty
\end{align*}
for every $0<\beta<\lambda_{1}\left(M\setminus B_{R_{0}}(x)\right)$. Hence
$$
\left| \int_{M}\bar{G}(x,y)\bar{f}(y)\,dy\right| < \infty.
$$
Then 
$$
\bar{u}(x):=\int_{M}\bar{G}(x,y)\bar{f}(y)\,dy
$$
solves $-\Delta \bar{u}=\bar{f}$, since $\bar{f}$ has zero average. Finally, the function
$$
u:=\bar{u}+\alpha \psi
$$
solves $-\Delta u = \bar{f}+\alpha\varphi = f$ and this concludes the proof of Theorem \ref{teo1}.

\end{proof}

\

\begin{proof}[Proof of Corollary \ref{cor-1}] In view of our assumptions, it follows that, for every $j_0\in \mathbb N$ there exists a constant $C=C(j_{0})$ such that, for any $j\geq j_{0}>1$, $\theta(j) =C\,j (2+j)^{\frac{\gamma}{2}}$. In particular $
\theta(j+1)-\theta(j)\sim C' j^{\frac{\gamma}{2}}$ as $j\to \infty$.  Moreover $\lim_{j\to\infty}\lambda_{1}\big(M\setminus B_{j}(p)\big)=\laess(M)>0$.  Thus we have
$$
b_{j}:=\frac{\theta(j+1)-\theta(j)}{\lambda_{1}\left(M\setminus B_{j}(p)\right)\zeta(j)} \leq \frac{C}{j^{1+\varepsilon}}.
$$
So the series $\sum_{j} b_{j}$ converges and the thesis follows from Theorem \ref{teo1}.
\end{proof}

\

\begin{proof}[Proof of Lemma \ref{lemma5}]
The choice of $m_{0}$ and Remark \ref{remark100} imply
\begin{equation}\label{eq300}
\mathcal{L}\left(0, \,2 a_{m_0}\right) \subset \mathcal{L}\left(0, \,A^{-1}\right)\subset M \setminus B_{1}(p).
\end{equation}
Let $z$ and $z_{0}$ be as in Lemma \ref{remark200}. Then
$$
G(p,z) \geq G(p,z_{0}) \exp \left(-C_{0}\sqrt{K(r(z)+1)}r(z)\right).
$$
From Lemma \ref{lemma1} we obtain
$$
G(p,z_0)\geq A^{-1}.
$$
In particular, by \eqref{eq300}, if $z\in \mathcal{L}\left(0, \,2 a_m\right)$, then
$$
\sqrt{K(r(z)+1)}r(z) \geq \theta(m).
$$
So,
\[ \theta(m)\leq \theta(r(z)+1))\,.\]
This yields, for every $m\in \mathbb N, m\geq 1$,
\[r(z) \geq m-1\,. \]
This concludes the proof of the lemma.
\end{proof}

\

\section{Cartan-Hadamard and model manifolds} \label{sec-ex}

We consider Cartan-Hadamard manifolds, i.e.~complete, non-compact, simply connected Riemannian manifolds with non-positive
sectional curvatures everywhere. Observe that on Cartan-Hadamard manifolds the cut locus of any point $p$ is empty.
Hence, for any $x\in M\setminus \{p\}$ one can define its polar coordinates with pole at $p$, namely $r(x) = \operatorname{dist}(x, p)$ and $\theta\in \mathbb S^{n-1}$. We have
\begin{equation*}
\textrm{meas}\big(\partial B_{r}(p)\big)\,=\, \int_{\mathbb S^{n-1}}A(r, \theta) \, d\theta^1d \theta^2 \ldots d\theta^{n-1}\,,
\end{equation*}
for a specific positive function $A$ which is related to the metric tensor, \cite[Sect. 3]{gri1}. Moreover, it is direct to see that the Laplace-Beltrami operator in polar coordinates has the form
\begin{equation*}
\Delta \,=\, \frac{\partial^2}{\partial r^2} + m(r, \theta) \, \frac{\partial}{\partial r} + \Delta_{\theta} \, ,
\end{equation*}
where $m(r, \theta):=\frac{\partial }{\partial r}(\log A)$ and $ \Delta_{\theta} $ is the Laplace-Beltrami operator on $\partial B_{r}(p)$.  We have
$$
m(r,\theta) =\Delta r(x).
$$

Let $$\mathcal A:=\left\{f\in C^\infty((0,\infty))\cap C^1([0,\infty)): \, f'(0)=1, \, f(0)=0, \, f>0 \ \textrm{in}\;\, (0,\infty)\right\} .$$ We say that $M$ is a rotationally symmetric manifold or a model manifold if the Riemannian metric is given by
\begin{equation*}\label{e2}
ds^2 \,=\, dr^2+\varphi(r)^2 \, d\theta^2,
\end{equation*}
where $d\theta^2$ is the standard metric on $\mathbb S^{n-1}$ and $\varphi\in \mathcal A$. In this case,
\begin{equation*}
\Delta \,=\, \frac{\partial^2}{\partial r^2} + (n-1) \, \frac{\varphi'}{\varphi} \, \frac{\partial}{\partial r} + \frac1{\varphi^2} \, \Delta_{\mathbb S^{n-1}} \, .
\end{equation*}
Note that $\varphi(r)=r$ corresponds to $M=\mathbb R^n$, while $\varphi(r)=\sinh r$ corresponds to $ M=\mathbb H^n $, namely the $n$-dimensional hyperbolic space. The Ricci curvature in the radial direction is given by
$$
\ricc( \nabla r, \nabla r) (x) = -(n-1)\frac{\varphi''(r(x))}{\varphi(r(x))}.
$$
Concerning the first eigenvalue of the Laplacian $\lambda_{1}(M\setminus B_{R}(p))$ we have the following lower bound.
\begin{lemma}\label{lemma-barta}
Let $(M,g)$ be a Cartan-Hadamard manifold with
$$
\ricc( \nabla r, \nabla r) (x)\leq -C\big(1+r(x)\big)^{\gamma}
$$
for some $C,\gamma>0$ and any $x\in M\setminus\{p\}$. Then
$$
\lambda_{1}\big(M\setminus B_{R}(p)\big) \geq C'R^{\gamma}
$$
for some $C'>0$.
\end{lemma}
\begin{proof} We can find $\varphi\in \mathcal{A}$ such that $\varphi(r)=\exp\big(B\,r^{1+\frac{\gamma}{2}}\big)$ for $r>1$, $B>0$ small  and $\ricc( \nabla r, \nabla r) (x) \leq -\frac{\varphi''(r(x))}{\varphi(r(x))}$. By the Laplacian comparison in a strong form, which is valid only on Cartan-Hadamard manifolds (see \cite[Theorem 2.15]{xin}), one has
$$
\Delta r(x) \geq \frac{\varphi'(r(x))}{\varphi(r(x))}.
$$
Hence
$$
\Delta r(x) \geq C r(x)^{\frac{\gamma}{2}} \geq C R^{\frac{\gamma}{2}} \quad\quad\hbox{in } M \setminus B_{R}(p).
$$
By a Barta-type argument (see e.g. \cite[Theorem 11.17]{gri2}),
\[\lambda_1(M\setminus B_R(p)) \geq [C R^{\frac{\gamma}{2}}]^2 \quad \textrm{in}\;\; M\setminus B_R(p)\,. \]
Thus, the thesis follows.
\end{proof}

\begin{proof}[Proof of Corollary \ref{cor-2}] By assumptions, it follows that, for every $j_0\in \mathbb N$, there exists a constant $C=C(j_{0})$ such that, for any $j\geq j_{0}>1$, $\theta(j) =C\,j (2+j)^{\frac{\gamma_{1}}{2}}$ and $
\theta(j+1)-\theta(j)\sim C' j^{\frac{\gamma_{1}}{2}}$ as $j\to\infty$.  Moreover, by Lemma \ref{lemma-barta} have $
\lambda_{1}\left(M\setminus B_{j-1}(p)\right) \geq C \,j^{\gamma_{2}}$ (Barta-type estimate). In particular $\laess(M)>0$. Thus we have
$$
b_{j}:=\frac{\theta(j+1)-\theta(j)}{\lambda_{1}\left(M\setminus B_{j-1}(p)\right)\zeta(j-1)} \leq \frac{C\,j^{\frac{\gamma_{1}}{2}}}{j^{\gamma_{2}+1+\frac{\gamma_{1}}{2}-\gamma_{2}+\varepsilon}}= \frac{C}{j^{1+\varepsilon}}.
$$
So the series $\sum_{j} b_{j}$ converges and the thesis follows from Theorem \ref{teo1}.
\end{proof}

In particular, in Corollary \ref{cor-2}, if $(M,g)$ is a model manifold with $\varphi\in\mathcal{A}$,
$$
\varphi(r)=\exp\left(r^{1+\frac{\gamma}{2}}\right), \quad \hbox{for } r>1,
$$
for some $\gamma>0$, $\laess(M)>0$, the assumptions on the Ricci curvature are satisfied with $\gamma_{1}=\gamma_{2}=\gamma$ and the hypothesis on $f$ reads
$$
|f(x)|\leq \frac{C}{\big(1+r(x)\big)^{\alpha}}
$$
for some
$$
\alpha>1-\frac{\gamma}{2}.
$$
These remark allows us to discuss the sharpness of the assumptions in Theorem \ref{teo1}. Indeed we show that the previous condition on $f$ is also necessary for the existence of a solution to the Poisson equation on the model manifold. In fact, as it has been shown in Section \ref{sec-proofs},
$$
\int_{M}G(x,y)f(y)\,dy<\infty \quad\quad\hbox{for any }\,  x \in M  \quad \Longleftrightarrow \quad \int_{M}G(p,y)f(y) \,dy<\infty .
$$
Hence a solution of $-\Delta u = f$ in $M$ exists if and only if
$$
u(p)=\int_{0}^{\infty}\left(\int_{r}^{\infty}\frac{1}{\varphi(t)^{n-1}}dt\right)f(r)\,\varphi(r)^{n-1}\,dr <\infty.
$$
With our choice of $\varphi$, by the change of variable $s=t^{1+\frac{\gamma}{2}}$, it is easily seen that, for any $r>0$ sufficiently large
$$
\int_{r}^{\infty}\frac{1}{\varphi(t)^{n-1}}dt \sim C r^{-\frac{\gamma}{2}}\exp\left(-(n-1)r^{1+\frac{\gamma}{2}}\right).
$$
Hence
\begin{align*}
 \frac 1{C} \int_{1}^{\infty} & r^{-\frac{\gamma}{2}}\exp\left(-(n-1)r^{1+\frac{\gamma}{2}}\right) \frac{1}{\big(1+r(x)\big)^{\alpha}}\exp\left((n-1)r^{1+\frac{\gamma}{2}}\right)\,dr  \leq |u(p)|\\&\leq C \int_{1}^{\infty} r^{-\frac{\gamma}{2}}\exp\left(-(n-1)r^{1+\frac{\gamma}{2}}\right) \frac{1}{\big(1+r(x)\big)^{\alpha}}\exp\left((n-1)r^{1+\frac{\gamma}{2}}\right)\,dr
\end{align*}
Therefore,
\begin{align*}
\frac 1 C\int_{1}^{\infty}\frac{1}{r^{\alpha+\frac{\gamma}{2}}}\,dr &\leq |u(p)|\leq C \int_{1}^{\infty}\frac{1}{r^{\alpha+\frac{\gamma}{2}}}\,dr\,.
\end{align*}
This yields that
$$|u(p)|<\infty \quad \textrm{
if and only if} \quad
\alpha>1-\frac{\gamma}{2}.
$$

\

\

\begin{ackn} The authors are members of the Gruppo Nazionale per l'Analisi Matematica, la Probabilit\`{a} e le loro Applicazioni (GNAMPA) of the Istituto Nazionale di Alta Matematica (INdAM). The first two authors are supported by the PRIN-2015KB9WPT Grant ``Variational methods, with applications to problems in mathematical physics and geometry''. 

The authors would like to express their gratitude to the referees for their very useful comments and suggestions.
\end{ackn}

\

\

\

\end{document}